\documentclass[10pt]{article}
\usepackage{amsmath,amsthm,amsfonts,amsopn,cite,mathrsfs}
\usepackage{fancybox,subfigure,latexsym} 
\usepackage{hyperref}

\theoremstyle{plain}
\newtheorem{theorem}{Theorem}
\newtheorem{proposition}{Proposition}
\newtheorem{lemma}{Lemma}
\newtheorem{corollary}{Corollary}
\theoremstyle{definition}
\newtheorem{definition}{Definition}
\theoremstyle{remark}
\newtheorem{remark}{Remark}

\usepackage[usenames]{color}
\usepackage[pdftex]{graphicx}

\newcommand{\Asp}{{\boldsymbol A}}     
\newcommand\Bisp{{\Bsp^1}}     
\newcommand{\Bsp}{{\boldsymbol B}}     
\newcommand\BispN{{(\Bisp, \, \|\ebbes\|^{(1)})}}     
\newcommand{\ebbes}{\mbox{$\,\cdot\,$}}     
\newcommand{\BspN}{(\Bsp, \, \|\ebbes\|_\Bsp)}     
\newcommand\Btsp{{\Bsp^2}}     
\newcommand\BtspN{(\Bsp^2, \, \|\ebbes\|^{(2)})} 
\newcommand\COiw{{\Csp^0_{1/w}}}     
\newcommand{\Csp}{{\boldsymbol C}}     
\newcommand{\CcG}{{\Ccsp(G)}}     
\newcommand{\Ccsp}{{\Csp_{\negthinspace c}}}     
\newcommand{\CcRd}{{\Ccsp(\Rst^d)}}     
\newcommand{\Rst}{{\mathbb R}}     
\newcommand\CuOiwRd{{\Csp^{\, 0}_{\nnth 1/w}(\Rdst)}}     
\newcommand{\nnth}{{ \negthinspace \: \negthinspace }}     
\newcommand{\Rdst}{{{\Rst^d}}}     
\newcommand\CuOiwsp{{\Csp^{\, 0}_{\nnth 1/w}}}     
\newcommand{\DPsi}{{\operatorname{D}_\Psi}}     
\newcommand\DPsimu{{D_\Psi \mu }}     
\newcommand{\DRd}{{\Dcsp(\Rst^d)}}     
\newcommand{\Dcsp}{{\boldsymbol{\mathcal D}}}     
\newcommand{\Drho}{{\operatorname{D}_{\rho}}}     
\newcommand\FBsp{{\FT \Bsp}}     
\newcommand{\FT}{{\operatorname{{\mathcal F} \negthinspace}}}     
\newcommand\FLtKats{{\Ltmtsp \nnth \cap \FT \Ltmisp}}     
\newcommand\Ltmtsp{{\Lsp^2_{m_2}}}     
\newcommand\Ltmisp{{\Lsp^2_{m_1}}}     
\newcommand{\Gsp}{{\boldsymbol G}}     
\newcommand{\LB}{ {\mathcal L} (\Bsp) }     
\newcommand{\LiRd}{{\Lisp \negthinspace (\Rst^d)}}     
\newcommand{\Lisp}{{\Lsp^1}}     
\newcommand{\LiRdN}{\big( \LiRd, \, \|\ebbes\|_1 \big)}     
\newcommand{\Lsp}{{\boldsymbol L}}     
\newcommand{\Liwsp}{{\Lsp^1_{\negthinspace w}}}     
\newcommand{\LiwRd}{{\Liwsp(\Rst^d)}}     
\newcommand{\LiwRdN}{\big( \LiwRd, \, \|\ebbes\|_{1,w} \big)}     
\newcommand{\LpRd}{{\Lpsp(\Rst^d)}}     
\newcommand{\Lpsp}{{\Lsp^p}}     
\newcommand{\LpRdN}{\big( \LpRd, \, \|\ebbes\|_p \big)}     
\newcommand\LpmRd{{\Lpmsp(\Rdst)}}     
\newcommand{\Lpmsp}{{\Lsp^p_{\negthinspace m}}}     

\newcommand{\LpwRd}{{\Lpwsp(\Rst^d)}}     
\newcommand{\Lpwsp}{{\Lsp^p_{\negthinspace w}}}     
\newcommand{\LpwRdN}{\big( \LpwRd, \, \|\ebbes\|_{p,w} \big)}     
\newcommand\LtKats{{\Ltmisp \nnth \cap \FT \Ltmtsp}}     
\newcommand{\LtR}{{\Ltsp(\Rst)}}     
\newcommand{\Ltsp}{{\Lsp^2}}     
\newcommand{\LtRN}{\big( \LtR, \, \|\ebbes\|_2 \big)}     
\newcommand{\LtRd}{{\Ltsp(\Rst^d)}}     
\newcommand{\LtRdN}{\big( \LtRd, \, \|\ebbes\|_2 \big)}     
\newcommand\LtvsTFd{{\Lsp^2_{v_s}(\TFd)}}     
\newcommand{\TFd}{{{ \Rdst \times \Rdsth }}}     
\newcommand{\Msp}{{\boldsymbol M}}     
\newcommand\nth{\negthinspace}     
\newcommand{\MiwRd}{{\Msp^1_{\nth w} (\Rdst)}}     
\newcommand\MiwRdN{{\nspb \MiwRd}}     
\newcommand{\Miwsp}{{\Msp^1_{\nth w}}}     
\newcommand\MpqmRd{{\Msp^{p,q}_{\nth m}(\Rdst)}}     
\newcommand\MpvsRd{{\Msp^p_{\nth v_s}\nth (\Rdst)}}     
\newcommand{\MspqRd}{{\Mspqsp(\Rst^d)}}     
\newcommand\Mspqsp{{\Msp^s_{\negthinspace p,q}}}     
\newcommand{\Nst}{{\mathbb N}}     
\newcommand\Psifam{{ \Psi = (\psi_i)_{i \in I} }}     
\newcommand\Psili{{(\Psi)_{|\Psi| \leq 1}}}     
\newcommand\QsRdN{\big( \QsRd, \, \|\ebbes\|_{\Qssp} \big)}     
\newcommand\Qssp{{\Qsp_s}}     
\newcommand\Qsp{ \boldsymbol Q }     
\newcommand\QsRd{{\Qssp(\Rdst)}}
\newcommand{\Rdsth}{{\widehat{\Rst}^d}}     
\newcommand{\Rtdst}{{\Rst^{2d}}}     
\newcommand{\ScPRd}{{\ScPsp(\Rst^d)}}     
\newcommand{\ScPsp}{{\Scsp'}}     
\newcommand{\Scsp}{{\boldsymbol{\mathcal S}}}     
\newcommand{\ScRd}{{\Scsp(\Rst^d)}}     
\newcommand{\Sp}{\operatorname{Sp}}     
\newcommand{\SpPsi}{\Sp_\Psi}     
\newcommand\SpPsif{{\SpPsi \nnth f}}     
\newcommand{\Strho}{{\operatorname{St}_{\negthinspace \rho}}}     
\newcommand{\Wsp}{{\boldsymbol W}}     
\newcommand{\checkm}{{^\checkmark}}     
\newcommand{\delo}{\delta > 0}     
\newcommand\delxii{{\delta_{\xi_i}}}     
\newcommand\epsfo{{\varepsilon/4}}     
\newcommand{\epso}{{ \varepsilon > 0 }}     
\newcommand\fBN{{ \|f\|_\Bsp}}     
\newcommand\fBn{{ \|f\|_\Bsp}}     
\newcommand\hatf{{\widehat{f}}}     
\newcommand{\hatg}{\widehat{g}}     
\newcommand\hatsi{{\widehat{\sigma}}}     
\newcommand\hkr{\hookrightarrow}     
\newcommand{\intRd}{\int_{\Rst^d}}     
\newcommand{\inv}{^{-1}}     
\newcommand\japy{{\langle y \rangle}}
\newcommand\japx{{\langle x \rangle}}     
\newcommand\limal{{\lim_{\alpha \to \infty} \, }}     
\newcommand{\lsp}{{\boldsymbol\ell}}     
\newcommand{\ltsp}{{\lsp^2}}     
\newcommand\psii{{\psi_i}}     
\newcommand\qandq{{ \quad \mbox{and} \quad }}     
\newcommand\sPsi{|\Psi|}     
\newcommand\sPsili{{|\Psi| \leq 1}}     
\newcommand\sPsitoz{{\, \sPsi \to 0 }}     
\newcommand{\spec}{\operatorname{spec}}     
\newcommand{\sumiI}{\sum_{i\in I}}     
\newcommand{\supp}{\operatorname{supp}}     
\newcommand\suth{{ \, | \, } }     
\newcommand\veps{{\varepsilon}}     
\newcommand{\wdash}{{ w^* \negthinspace \mbox{-}}}     
\newcommand{\wwst}{{\wdash \wdash}}     
\newcommand\xii{{\xi_i}}     

\def\citeX{\cite}                       

\def\kg{{k_g}}
\def\kf{{k_f}}

\def\opnorm#1#2{{|\nnth \| {#1} | \nnth \|_{#2} \, }}
\def\normta#1#2{{  \| {#1}   \|_{#2} \, }}
\def\japarg#1{{\langle #1 \rangle}}
\def\nspb#1{{ (#1,\| \ebbes \|_{#1} )}}
\def\NSPB#1{{ (#1,\| \ebbes \|_{#1} )}}
\def\Bnorm#1{{ \|#1\|_\Bsp}}
\def\normB#1{{ \| #1 \|_\Bsp }}     

\newcommand{\footremember}[2]{%
	\footnote{#2}
	\newcounter{#1}
	\setcounter{#1}{\value{footnote}}%
}

\providecommand{\keywords}[1]
{
	\small	
	\textbf{\textit{Keywords:}} #1
}
\providecommand{\subjclass}[2]
{
	\small	
	\textbf{\textit{2010 Mathematics Subject Classification.}} #2
}

\begin{document}

\title{ Completeness of shifted dilates in invariant Banach spaces\\ of tempered distributions}
\author{
	Hans G. Feichtinger\footremember{Vienna}{NuHAG, Faculty of Mathematics, University of Vienna, Oskar-Morgenstern-Platz 1,1090 Vienna, AUSTRIA,
\newline  email: \tt hans.feichtinger@univie.ac.at }
	\and Anupam Gumber \footremember{Indian}{Department of Mathematics, Indian Institute of Science,560022 Bangalore, INDIA,
	email: \tt anupamgumber@iisc.ac.in}
}
\maketitle
\begin{abstract}
We show that  well-established methods from the theory of Banach modules and time-frequency analysis allow to derive completeness results for the collection of shifted and dilated version of a given (test) function in a quite general setting. While the basic ideas show strong similarity to the arguments used in a recent paper by  V.~Katsnelson
 we extend his results in several directions, both relaxing the assumptions and widening the range of applications. There is no need for the Banach spaces considered to be embedded into $\LtRN$, nor is the Hilbert space
structure relevant. We choose to present the results in the setting of
the Euclidean spaces, because then the Schwartz space $\ScPRd$ ($d \geq 1$)
of tempered distributions provides a well-established environment for
mathematical analysis. We also establish connections to modulation
spaces and   Shubin classes $\QsRdN$, showing that
they are  special cases of Katsnelson's setting (only) for $s \geq 0$.
\end{abstract}
\\
{Primary 43A15, 41A30, 43A10, 41A65, 46F05, 46B50; \newline Secondary 43A25, 46H25, 46A40}

\section{Introduction}

The motivation for the present paper lies in the study of \citeX{ka19-1}, which shows that the set of all shifted, dilated Gaussians is total in certain translation and modulation invariant Hilbert spaces of functions which are  continuously embedded into $\LtRN$.

By working in a more general setting we show that the results presented
in that paper can be extended in various directions. During the studies
it turned out that the setting used in \citeX{dipivi15-1} appears to
be most appropriate, although it is mainly the existence of a double
module structure (namely with respect to convolution and pointwise multiplication) which makes the key arguments work. Such Banach spaces have been discussed already long ago by the first author, under the name of {\it standard spaces} in order to study compactness in function spaces of distributions, or in order to derive a number of module theoretical properties of such spaces, as given in \citeX{brfe83}.
The setting which we choose 
is also closely related to Triebel's systematic work
concerning the {\it theory of function spaces}.


The paper is organized as follows. After providing basic notations we
will describe the setting of invariant function spaces of tempered
distributions. Since we are addressing the question of totality of
a set of shifted and dilated version of a given test function we have
to restrict our attention to Banach spaces which contain $\ScRd$ as a
dense subspace. On the other hand it is convenient and still very
general to work within the realm of tempered distributions. While the
methods employed are valid in a more general context this setting should
make the reading of the paper easier for the majority of readers.


In preparation of the main result we then go on to prove some
technical results concerning the approximation of convolution products
in {\it Beurling algebras}. Subsequently we will derive our main result.
The remaining sections will be devoted to an exploration
of the wide range of applicability of the result presented. In the
final section we will explain, why our results contain the key
results of Katsnelson's paper and in which sense we are going (far)
beyond the setting described in his paper.

\section{Notations and Conventions}

For a bounded linear operator $T$ on a Banach space $\BspN$ we write
$T \in \LB$,  and use the symbol $\opnorm T \Bsp$ in order to describe
its operator norm
$$  \opnorm T \Bsp = \sup_{\fBN \leq 1} \normta {T(f)} \Bsp. $$
We will make use of standard facts concerning tempered distributions.
Recall that $\DRd$, the subspace of smooth function with {\it compact
support} is a dense subspace of  the Schwartz space $\ScRd$ of {\it rapidly decreasing functions} $\ScRd$ with the standard topology.

The Fourier invariance of $\ScRd$ allows
to extend the classical Fourier transform in a unique ($\wwst$-continuous)
form by the  rule
$\hatsi(f) = \sigma(\hatf), f \in\ScRd$, for any $\sigma \in \ScPRd$.
This {\it extended Fourier transform} provides an automorphism of $\ScPRd$, and hence for any Banach space $\BspN$ continuously embedded into $\ScPRd$,
we use the symbol $\BspN \hkr \ScPRd$,  i.e.\ satisfying
\begin{equation}\label{contembscp1}
  \normta {f_n - f} \Bsp \to 0 \,\, \mbox{in} \,\, \BspN \,\, \mbox{for} \,\, n \to \infty  \quad
 \Rightarrow  \quad  f_n(g) \to f(g), \,\,\, \forall g \in \ScRd 
\end{equation}
also $\FT \Bsp$ is a well defined Banach space of distributions
with the natural norm $\|\hatf\|_{\FBsp} = \|f\|_\Bsp$, $f \in \Bsp$.

In addition to the usual function spaces such as $\LpRdN$, with
$1 \leq p \leq \infty$ we also need their weighted versions. Given
as strictly positive weight $w(x)>0 $ we obtain   Banach spaces
$$\LpwRd := \{ f  \suth fw \in \LpRd \}, \quad \mbox{resp.} \,\,\,
\Bsp_w = \{f \suth  fw \in \Bsp \}
$$
with the  norm $\|f\|_{p,w} = \normta f \LpwRd = \normta {f\,w} \LpRd$.

We will be only interested in translation invariant function spaces of this
form, hence we restrict our attention to {\it moderate} weight functions, which (without loss of generality) can be assumed to be strictly positive and continuous.

On such spaces every translation operator, defined   via
$T_z f(x) = f(x-z)$ is bounded   on $\LpwRdN$ and
the operator norm $w(z) := \opnorm {T_z} \Bsp$ is a so-called {\it Beurling
weight function}, i.e.\ a strictly positive {\it submultiplicative function}
satisfying
\begin{equation}\label{Beurlwgt1}
  w(x+y) \leq w(x) \,
  w(y), \quad x,y \in \Rdst.
\end{equation}
Such weight functions generate weighted $\Lisp$-spaces which are
Banach algebras with respect to {\it convolution} (so-called {\it
Beurling algebras}), due to  the pointwise  (a.e.) estimate
\begin{equation} \label{ptwconvw1}
  |f \ast g| w \leq  |f|w \ast |g|w, \quad f,g \in \LiwRd,
\end{equation}
which imply the norm estimate
\begin{equation} \label{ptwconvw2}
  \normta  {f \ast g} {{1,w}}  \leq \normta  {f} {{1,w}}  \normta  { g} {{1,w}},
  \quad f,g \in \LiwRd.
\end{equation}
Details concerning these so-called {\it Beurling algebras} are found in
Reiter's book \cite{re68} (or \cite{rest00}), Ch.2.6.3.
Among others they are translation invariant, with the property
\begin{equation}\label{transBeurl1}
  \|T_x f\|_{1,w} \leq w(x) \|f\|_{1,w}, \quad \forall x \in \Rdst, f \in \LiwRd,
\end{equation}
and that shifts depend continuously on $x$, meaning that $x \to T_x f$
is a continuous mapping from $\Rdst$ to $\LiwRdN$, or equivalently
\begin{equation}\label{transcontw2}
  \lim_{x \to 0}  \| T_x f - f\|_{1,w} = 0, \quad \forall f \in \LiwRd.
\end{equation}
Occasionally we will work with the Banach convolution algebra
 $\MiwRdN$, the space of Radon measures $\mu$
such that $ w \mu$ is a bounded measure, or equivalently, the dual
space of $\nspb \CuOiwRd$. It contains $\LiwRdN$ as a closed ideal.

It is a general fact that a positive, continuous weight function
$m$ is {\it moderate} ({\it  with respect to} $w$) if and only if it satisfies
\begin{equation}\label{moddef04}
  m(x+y) \leq m(x) \,  w(y), \quad x,y \in \Rdst
\end{equation}
for some submultiplicative weight function $w$ (\citeX{fe79,gr07}).
In this note we will concentrate on {\it polynomially moderated}
weights, i.e.\ weights for which one can use
\begin{equation} \label{japdefs}
w(x)= w_s(x):= \japx^s = (1+ |x|^2)^{1/2} \approx (1+|x|)^s,\quad s \geq 0.
\end{equation}
Any function $ \japx^s, s \in \Rst$
is a moderate weight, with respect to $\japx^{|s|}$.
Consequently the weighted, translation invariant
spaces $\LpmRd$ will be continuously embedded into $\ScPRd$.

Beurling algebras share several important properties with $\LiRdN$:
Compactly supported functions are dense, and there are
bounded approximate identities in $\LiwRdN$, i.e.\ bounded nets
$(e_\alpha)_{\alpha \in I}$ such that for each $h \in \LiwRd$:
\begin{equation}\label{apprid03}
  \| e_\alpha \ast h - h \|_{1,w} \to 0 \,\, \mbox{for} \,\, \alpha \to \infty.
\end{equation}
The boundedness of such a family also allows to extend this property to
relatively compact sets, hence one has:
For every relatively compact set $M \subset \BspN$ and $\epso$ one
can find some index $\alpha_0$ such that  $\alpha \succ \alpha_0$
implies:
\begin{equation} \label{compappr1}
 \normta { e_\alpha \ast h - h}  {1,w} \leq \veps, \quad \forall h \in M.
\end{equation}

For us,  approximate identities obtained by compression of a given function
$g \in\LiwRd$ with $\hat{g}(0) \neq 0$ will be important. Without loss of
generality let us assume that $\hat{g}(0) = \intRd g(x)dx = 1$.
We use the $\Lisp$-isometric compression
\begin{equation}\label{Strohdef03}
g_\rho(x) =  \Strho g(x) = \rho^{-d}  g(x/\rho), \quad \rho > 0,
\end{equation}
with $\supp(\Strho g) = \rho \supp(g)$ and
$ \normta {g_\rho} \Lisp =
\normta {\Strho g} \Lisp = \normta g \Lisp$ for $g \in \LiRdN$.

For any {\it radial symmetric},  {\it increasing} weight $w$   
satisfying $w(y) \leq w(x)$ if $|y| \leq |x|$, one has 
\begin{equation}\label{Strohest1}
  \normta  {g_\rho}  {1,w} =
  \normta {\Strho g} {1,w} \leq
  \normta {g} {1,w} \quad \mbox{for all} \,\, \rho \in (0,1).
\end{equation}
Without loss of generality we will make this assumption concerning $w$,
satisfied by all the usual examples. In the general case it is possible to
replace a given weight function by another one satisfying this
condition. 
With this extra condition
we  obtain bounded approximate identities by compression:
\begin{lemma} \label{StrhoDirac1}
For any $f \in \LiwRd$ and $g \in \LiwRd$ with $\hat{g}(0)=1$
one has
\begin{equation}\label{apprid04}
  \lim_{\rho \to 0}  \, \normta  {\, g_\rho \ast f - f}  {1,w} =
  \lim_{\rho \to 0}  \, \normta  {\, \Strho g \ast f - f}  {1,w} = 0. 
\end{equation}
\end{lemma}
Polynomial weights $w_s$
satisfy the so-called {\it Beurling-Domar condition} (see \citeX{rest00})
and hence
$$ \{ f \suth f \in \LiwRd,  \spec(f)  = \supp(\hatf) \, \,\, \mbox{is compact} \, \} $$
the subset of all {\it band-limited} elements
is a dense subspace of $\LiwRdN$.

For two topological vector spaces $\Bisp$ and $\Btsp$ we will write
$\Bisp \hookrightarrow \Btsp$ if the embedding is continuous\footnote{In contrast
to \citeX{dipivi15-1} we do not assume density of the embedding whenever we
use this symbol. We rather prefer to put this as an explicit extra assumption.}. If both
of them are normed spaces this means of course that there exists
some constant $C > 0$ such that $\normta f \Btsp \leq C \normta f \Bisp$,
for all $f \in \Bisp \subset \Btsp$. For Banach spaces continuously
embedded into $\ScPRd$ the boundedness of an inclusion mapping follows
from the simple inclusion $\Bisp \subseteq \Btsp$, via the Closed Graph Theorem.

For the rest of this paper we will work with
the following {\bf standard assumptions}, similar
to the setting chosen in \cite{dipivi15-1}:
\begin{definition} \label{mintempstanddef}
A Banach space $\BspN$ is called a {\it minimal tempered standard space}
(abbreviated as {\bf MINTSTA})
if the following conditions are valid:
\begin{enumerate} \item One has the following sandwiching property:
\begin{equation}\label{ScSandw}
   \ScRd \hookrightarrow \BspN \hookrightarrow  \ScPRd;
\end{equation}
\item
$\ScRd$ is dense in $\BspN$ (minimality);
\item $\BspN$ is translation invariant, and for
some $n_1 \in \Nst$ and $C_1 > 0 $ one has
\begin{equation}\label{transl1}
  \|T_x f\|_\Bsp \leq C_1 \japx^{n_1} \|f\|_\Bsp \quad  \forall  x \in \Rdst;
\end{equation}
\item  $\BspN$ is modulation invariant, and for
some $n_2 \in \Nst$ and $C_2 > 0 $ one has
\begin{equation}\label{modul1}
  \|M_y f\|_\Bsp \leq C_2 \japy^{n_2} \|f\|_\Bsp \quad \forall y \in \Rdst.
\end{equation}
 \end{enumerate}
\end{definition}   

\begin{remark}
The notion of MINTSTAs relates the approach to the use of ``standard
spaces'' in the work of the first author, starting in the 70th, specifically
the use of double module properties in \citeX{brfe83}.

The formal definition provided here is inspired by the work \citeX{dipivi15-1}, where such spaces are called TMIBs. The density of
$\ScRd$ is a part of their definition. The few interesting spaces
which do not satisfy this extra condition are typically DTMIBs in their
terminology (dual translation, modulation invariant Banach spaces). See also
\cite{dipiprvi19} for a more general setting.
\end{remark}

\begin{remark}
The term {\it Banach spaces in standard situation}  has been used
in a number of papers of the first author, e.g.\ in order to prove
results about compactness in such spaces (\citeX{fe84}), in order
to introduce Wiener amalgam spaces (\citeX{fe83}), or in order to study
spaces with a double module structure (\citeX{brfe83}). In each of these
cases it is important that it is meaningful for the objects under
consideration (functions, measures or distributions) to allow pointwise
products with suitable test functions (leading to a localization), and
make use of this fact that this is definitely possible for elements
in the dual of a space $\Asp_c = \CcRd \cap \Asp$, of compactly
supported test functions, where $\Asp$ is a suitable pointwise
Banach algebra of test functions which is also translation invariant.
\end{remark}

\begin{remark}
The situation described in Definition \ref{mintempstanddef} are special
cases of this more general notion of {\it standard spaces}, with the
main restriction (more or less made for the convenience of the reader,
and in accordance with \citeX{dipivi15-1}) that we assume $\BspN \hkr \ScPRd$.

One should also observe that for a non-trivial Banach space of tempered
distributions satisfying the invariance properties (under translation
and modulation) one always has the continuous embedding $\ScRd \hkr \BspN$.
In fact, given properties 3. and 4. in Definition \ref{mintempstanddef}
one can show that there is a minimal space in the corresponding family of spaces,
namely $\Wsp(\FT \Lsp^1_{v_{n_2}},\lsp^1_{v_{n_1}})$ according to
\citeX{fe87-1}, and hence  the following chain of inclusions is valid:
\begin{equation}\label{ScMinEmb}
  \ScRd \hkr  \Wsp(\FT \Lsp^1_{v_{n_2}},\lsp^1_{v_{n_1}}) \hkr \BspN.
\end{equation}
The minimality condition ensures that these embeddings are dense
embeddings.
\end{remark}
The following result 
is stated  for later reference. The proof is left to the reader as an exercise.
\begin{proposition} \label{InvarIntersec}

\noindent (i) For any MINTSTA $\BspN$ also its Fourier version
  $\FBsp = \{ \hatf \suth f \in \Bsp\}$
  is a MINTSTA  with respect to the natural norm
\begin{equation} \label{FBnorm1}
\normta  \hatf \FBsp  =  \fBN, \quad f \in \Bsp.
\end{equation}

\noindent (ii)  Given two MINTSTAs $\BispN$ and $\BtspN$,  also their
intersection (or their sum) is a MINTSTA, with the corresponding
natural norms, e.g.
$$ \|f\|_{\Bisp \cap \Btsp} :=  \|f\|_\Bisp + \|f\|_\Btsp,
  \quad f  \in \Bisp \cap \Btsp.$$
\end{proposition}

\subsection{Equivalent Assumptions}

Let us first discuss a few alternative assumptions which lead
to the same family of spaces.
\begin{lemma} \label{invarprops1}

\noindent
i) Assume that $\BspN$ is a Banach space satisfying conditions 1. and 2. of Definition  \ref{mintempstanddef}. Then 3. and 4. together are equivalent to the
claim that the space $\BspN$ is invariant under TF-shifts
$\pi(z) = M_y T_x$, with $z = (x,y) \in \Rtdst$,
and that for some constant $C_3 > 0$ and $s \geq 0$ one has:
\begin{equation} \label{pizestim1}
\opnorm {\pi(z)} \Bsp  \leq C_3 {\japarg z }^s,
\quad \forall z \in \Rtdst
\end{equation}
or equivalently described:
\begin{equation} \label{pizestim2}
\normta {\pi(z)f } \Bsp  \leq C_3 {\japarg z }^s \fBN
\quad \forall z \in \Rtdst, \forall f \in \Bsp.
\end{equation}

\noindent
ii) For any {\it minimal tempered standard space} one has: for any $g \in \Bsp$
\begin{equation}\label{contpiz1}
   z \mapsto \pi(z)g \,\,\mbox{ is continuous from} \,\, \Rtdst \,\, \mbox{to} \,\, \, \BspN.
\end{equation}

\noindent
iii) Conversely, assuming that a Banach space   $\BspN$ continuously
embedded into $\ScPRd$ satisfies (\ref{pizestim2}) and (\ref{contpiz1}).
Then (\ref{transl1}) and (\ref{modul1}) are valid, for $n_1 = s = n_2$.
Moreover, $\ScRd$ is embedded into $\BspN$ as a dense subspace\footnote{This fact justifies the use of the word minimality.}.
\end{lemma}

\begin{proof}
i) It is clear that the estimates (\ref{transl1}) and (\ref{modul1})
are just special cases of (\ref{pizestim1}), e.g. for $n_1 = n = n_2$
being any integer $n$ with $s \leq n$.

Conversely assume that the two estimates (\ref{transl1}) and (\ref{modul1})
are valid, and that we have to estimate the norm of $\pi(z)g$ in $\BspN$.
Clearly
\begin{equation}\label{pizestim02}
  \normta {\pi(z)g} \Bsp = \normta {M_y T_x g} \Bsp \leq
    \opnorm  {M_y} \Bsp \opnorm {T_x} \Bsp \|g\|_\Bsp
    \leq  C_1 C_2 \japarg{y}^{n_2} \japarg{x}^{n_1} \normta g \Bsp.
\end{equation}
By choosing $s = n_1 + n_2 $ we obtain for $C_3 = C_1 C_2$:
\begin{equation}\label{pizestim03}
  \normta {\pi(z)g} \Bsp
    \leq  C_1 C_2 \japarg{z}^{n_2} \japarg{z}^{n_1} \normta g \Bsp
    \leq C_3 \japarg{z}^s.\end{equation}

ii) The continuous shift property is clear for $g \in \ScRd$ in the
Schwartz topology. Due to the continuous embedding of $\ScRd$ into
$\BspN$ condition (\ref{contpiz1}) is valid for $g \in \ScRd$. Using
the (uniform) boundedness of TF-shifts with say $|z| \leq 1$ it follows
easily that (\ref{contpiz1}) is valid for any $g \in \Bsp$ by the
usual approximation argument.

iii)
The continuity of translation   implies that every
element $g \in \BspN \subset \ScPRd$ can be regularized, i.e. it
can be approximated by functions in the Schwartz space, because the
usual regularization procedures of the form $ \sigma \mapsto
\Strho g_0 \ast(\Drho g_0 \cdot g)$ map $g$ into $\ScRd$, but
also approximate $g$ in $\BspN$. Details are found
 in \cite{brfe83}, where the closure of $\ScRd$ is
characterized as 
 $\Bsp_{\nnth \Asp \Gsp} = \Bsp_{\nnth \Gsp \Asp}$.
 %
\end{proof}

     When comparing with the setting of \citeX{dipivi15-1} we have
     the following connection:
\begin{lemma} \label{charminFOUSS}
A Banach space $\BspN$ is a  minimal 
tempered Fourier standard space
if and only if $\BspN$ as well as its Fourier image
$\FT \Bsp$, with the norm $\|\hatf\|_{\FT \Bsp} = \fBn $
are translation invariant Banach spaces of distributions
in the sense of \cite{dipivi15-1} containing $\DRd$ as a dense
subspace.
\end{lemma}

\begin{proof}
According to Theorem 1 of \citeX{dipivi15-1} a {\it translation invariant
Banach space of tempered distributions} satisfies conditions
1. to 3. of our definition. Being sandwiched between $\ScRd$
and $\ScPRd$ its Fourier transform  $\FBsp$ is a well-defined
Banach space, which is itself again a Banach space in sandwich position.

The fact that translation on the Fourier transform side corresponds
to modulation on the time-side (combined with the corresponding)
implies immediately that the validity of 4. in Definition  \ref{mintempstanddef}
is equivalent to a polynomial estimate of the translation operator
for $\nspb \FBsp$.
\end{proof}

\begin{remark}
It is noteworthy to mention that the
sandwiching properties  1. and 2. above follow often from 3. and 4., e.g.
if $\BspN$ is a solid BF-space containing
 $$\CcRd = \{k \suth k  \,  \mbox{continuous, complex
valued on} \, \Rdst, \mbox{with} \supp(k) \, \mbox{compact} \}$$
as a dense subspace and satisfying 3.
(because 4. above is trivial for solid  spaces, see \citeX{brfe83,fe79}).
\end{remark}

By a slight adaptation of the terminology of  Y.~Katznelson
 (see \cite{ka76}) we call a Banach space $\BspN$ in ``sandwich position''
a {\it homogeneous Banach space of tempered distributions} if
\begin{enumerate}
  \item Translations are isometric on $\BspN$:
  $ \normta {T_xf} \Bsp = \normta  f \Bsp, \quad \forall f \in \Bsp; $
  \item translation is continuous, i.e.
  $ \lim_{x \to 0} \|T_x f - f\|_\Bsp = 0 \quad \forall f \in \Bsp. $
\end{enumerate}

\section{Discretization of convolution in Beurling algebras}

For the rest we assume that the weight function $w$ is not only a
continuous and submultiplicative function on $\Rdst$, but in addition
that it is radial symmetric, with increasing profile,
i.e.\ $w(y) \leq w(x)$ whenever $|y| \leq |x|, x,y \in \Rdst$.
This is no loss of generality, because any general weight function
(of polynomial growth) is dominated by another submultiplicative function
with this extra property. The main advantage of this assumption is the
fact that it implies that the dilation operator $g \mapsto \Strho g$
is non-expansive on $\LiwRdN$ as well as on $\NSPB \MiwRd$ for $\rho \in (0,1)$.

In short, based on a variant of the key result of \citeX{fe16},
$\BspN$ is a {\it Banach module over} $\LiwRdN$, and in fact over
$\NSPB \MiwRd$. Hence we have
\begin{equation}\label{MiwModul1}
   \| \mu \ast f\|_\Bsp \leq \|\mu\|_\Miwsp \|f\|_\Bsp, \quad
   \forall  \mu \in \MiwRd, f \in \Bsp.
\end{equation}
Let us not forget to mention the validity of the innocent looking
associative law:
\begin{equation} \label{assocBanmod1}
   (\mu_1 \ast \mu_2) \ast f = \mu_1 \ast (\mu_2 \ast f), \quad \mu_1,\mu_2 \in \MiwRd, f \in \Bsp.
\end{equation}

This result concerning ``integrated group representations'' can be
considered a folklore result, see for example
\cite{bo04-5}, Chap.8,  
working for general group representations on a Banach space.  Similar
results are given in \citeX{brfe83} and  \citeX{dipivi15-1}, for example.

A minor modification of the results in \citeX{du74} gives the following
characterization:
\begin{lemma} \label{essModChar1}
Let $\BspN \hkr \ScPRd$ be a {\it Banach convolution module}
over some weighted measure algebra $\MiwRdN$. Then, viewed as
a Banach module over the corresponding {\it Beurling algebra} $\LiwRdN$
it is an essential Banach module if and only if translation is continuous
in $\BspN$, i.e.
\begin{equation}\label{contshiftB1}
  \|T_x f - f \|_\Bsp \to 0 \quad \mbox{for} \,\, x \to 0, \forall f \in \Bsp
\end{equation}
or equivalently
\begin{equation}\label{essModAU1}
  \limal \|e_\alpha \ast f - f \|_\Bsp = 0 \quad \forall f \in \Bsp
\end{equation}
for any bounded approximate identity $(e_\alpha)_{\alpha \in I}$
in $\LiwRdN$.
\end{lemma}
\begin{proof}
Since $\delta_x \in \MiwRd$ for any $x \in \Rdst$ it is clear
that $\Bsp$ is translation invariant and: 
\begin{equation}\label{MiwTrans1}
  \|T_x f\|_\Bsp = \|\delta_x \ast  f\|_\Bsp \leq \|\delta_x\|_\Miwsp \|f\|_\Bsp    = w(x) \|f\|_\Bsp
  \quad \forall f \in \Bsp, x \in \Rdst.
\end{equation}
If translation is continuous in $\BspN$  (see \ref{contshiftB1})
the usual approximate units (convolution with $\Lisp$-normalized
bump functions with small support) act as expected, i.e.\ for $\epso$
there exists some $h \in \LiwRd$ such that
$$ \| h \ast f -f \|_\Bsp < \varepsilon.$$
Hence  obviously $\BspN$ is an {\it essential Banach module} over
$\LiwRdN$, i.e. (by the definition) that the linear span of
$\LiwRd \ast \Bsp$ is dense in $\BspN$.

The equivalence to both stated properties (namely (\ref{contshiftB1})
and (\ref{essModAU1})) follows therefrom.
\end{proof}

As a special case which will be used frequently in the sequel we have
\begin{corollary} \label{LiwINMiw1}
$\LiwRdN$ is a closed ideal in $\MiwRdN$, i.e.\ one has
\begin{equation}\label{LiwastMiw}
  \normta {\mu \ast f}  {\Liwsp} \leq
   \normta {\mu }  {\Miwsp}   \normta { f}  {\Liwsp},
\quad \forall \mu \in \MiwRd, \forall f \in \LiwRd.
\end{equation}
$\LiwRd$ consists exactly of those
elements in $\MiwRd$ which have the continuous shift property.

Moreover, any bounded linear operator $T$ on $\LiwRd$ which commutes
with all translations is of the form $T(f) = \mu \ast f$, for a
uniquely determined $\mu \in \MiwRd$.
\end{corollary}

\begin{proof}
Since $\LiwRd$ is an   $\MiwRd$-module it is clear
that we have the first two statements. The isometric embedding of
$\LiwRdN$ (with the weighted norm $\|fw\|_\Lisp$) into $\MiwRdN$ is
a routine task. The subspace of measures in $\MiwRd$ with continuous
shift form an essential $\LiwRd$-module. Hence these elements
can be approximated by elements of the form $e_\alpha \ast \mu
\in \LiwRd \ast \MiwRd \subseteq \LiwRd$ (cf. \citeX{du74}).

The additional statement about multipliers is just a reformulation
of the main result of Gaudry (\nth \citeX{ga69}), provided  a complementary
perspective. This correspondence is in fact an isometric one.
\end{proof}

For the technical part of our proof we need the following joint
estimate on the discretization operators $\DPsi$, showing their uniform
boundedness over the family $\Psili$. 
\begin{lemma} \label{DPsiunifest2}
Given a Beurling weight $w$ on $\Rdst$, there is a
uniformly estimate for the family of discretization operators with
respect to BUPUs of size $|\Psi| \leq 1$. For some $C_1 > 0$ one has:
\begin{equation}\label{DPsiunifest3}
 \sum_{i \in I} |\mu(\psi)| w(x_i) =
  \| \DPsimu\|_{1,w} \leq C_1 \|\mu\|_\MiwRd, \quad \forall \mu \in \MiwRd.
\end{equation}
\end{lemma}
\begin{proof}
We will use  that $\MiwRdN$ is the dual space of $\nspb \CuOiwsp$,
with the natural norm  $ \|f/w\|_\infty$.
Using the density of $\CcRd$ in  $\nspb \CuOiwsp$ we first verify the adjointness relation
\begin{equation}\label{DPsidual}
  [\DPsimu](f) = \mu(\SpPsif), \quad \forall f \in \CcRd,
\end{equation}
justified by
\begin{equation}\label{Dpsidual1}
  \SpPsi^* \mu(f) =  \mu (\SpPsif ) = \mu \left (\sumiI f(\xii) \psii \right) =    \left(  \sumiI \mu(\psi_i) \delxii \right ) (f) = \DPsimu(f),
\end{equation}
and look for an estimate of $f \mapsto \SpPsif$  on $\nspb \CuOiwRd$.

Given the (continuous) weight function $w$ we  set
$C_1 = max_{|z| \leq 1} w(z)$. Then for any BUPU
$\Psi$ with $|\Psi| \leq 1$    we have,
using $\supp(\psi_i) \subseteq B_1(\xii)$ for each $i \in I$
\begin{equation}\label{wgtestim}
  1/w(x) \leq  w(\xii-x)/w(\xii)  \leq C_1/ w(\xii), \quad x \in \supp(\psi_i),
\end{equation}
and consequently the following pointwise estimate for any $f \in \CuOiwRd$:
\begin{equation}\label{wgtestim2}
 |\SpPsif(x)|/w (x) \leq \sumiI  |f(x_i)|\psi_i(x) w(x)
 \leq C_1 \sumiI [|f(\xii)|/w(x_i)] \psi_i(x) \leq  C_1 \|f/w\|_\infty
\end{equation}
or in terms of the norm on $\CuOiwRd$:
\begin{equation}\label{SpPest2}
  \| \SpPsif\|_\COiw \leq C_1 \|f\|_\COiw, \quad f \in \CuOiwRd,
\end{equation}
respectively expressed by operator norms:
\begin{equation}\label{SpestOPN1}
  \opnorm {\DPsi} {\MiwRd} = \opnorm {\SpPsi} {\CuOiwRd} \leq C_1,\quad \forall |\Psi| \leq 1.
\end{equation}
\end{proof}

Next we show that a convolution product
within a Beurling algebra $\LiwRdN$ can be discretized, i.e.\
a convolution product can be approximated by a finite linear
combination of shifted version of either convolution factor.
This result is inspired by Chap.1.4.2 of Reiter's book \citeX{re68}
and can be viewed as a variant of Theorem 2.2 in \cite{fe77-2}.
\begin{theorem} \label{discrconvBeurl1}
Given two functions $g,f $ in some Beurling algebra $\LiwRdN$
and $\epso$ there exists   $\delo$ such that one has for
any $\Psifam$ with $\sPsi \leq \delta$:
\begin{equation} \label{convappr02}
\normta {g \ast f - g \ast \DPsi f} \LiwRd < \veps. 
\end{equation}
\end{theorem}

\begin{remark}
This result is closely related to the compactness criteria
for function spaces \citeX{fe82-1} and \citeX{fe84}. It is
clear that - by the tightness and boundedness in $\LiwRdN$ - of the
family $\DPsi f, \sPsili$, also $g \ast \DPsi f$ is a bounded
and tight family. It is also clear that it is equicontinuous
in $\LiwRdN$, since we can control the shift error as follows:
\begin{equation}\label{equicont02}
  \normta {g \ast \DPsi f  - T_z( g \ast \DPsi f)} \Liwsp
   \leq \normta {g   - T_z g  } \Liwsp  \cdot \normta {\DPsi f} \Miwsp
  \leq \normta {g   - T_z g  } \Liwsp  \cdot C_1 \normta f \Liwsp,
\end{equation}
which tends to zero for $z \to 0$, since translation is continuous
in any Beurling algebra $\LiwRdN$ (see \citeX{re68}, Chap.1,6.3.).

According to \citeX{fe82-1} this implies that this set is relatively
compact in $\LiwRdN$, and hence there is a {\it subsequence} which
converges in the norm.  {\it However, we want to prove actual
convergence of the net}, for $\sPsitoz$, not only for a subsequence.
\end{remark}

\begin{proof} We start by fixing $\epso$ and assume that
$f,g \in \LiwRd$ are given.
For simplicity we assume without loss of generality that
both $g,f$ are normalized in $\LiwRdN$, in order to make
the presentation more straightforward.

We will prove the estimate by reduction
to the dense subspace $\CcRd$ of $\LiwRdN$.
First,  we choose $\kg,\kf \in \CcRd$
such that $\normta {g - \kg} \LiwRd < \eta$ and
$ \normta{f - \kf} \LiwRd < \eta$ for some
  $\eta \in (0,\veps/(12 C_1))$.
Since convolution is a continuous, bilinear operation
in the Banach convolution algebra $\LiwRdN$
we can choose $\eta >0$ such that in addition
\begin{equation}\label{firstest1}
  \normta{ g \ast f - \kg \ast \kf} \LiwRd < \veps/4.
\end{equation}
As now $\kg,\kf \in \CcRd$ the convolution product $\kg \ast \kf$ also
has compact support, but also all the functions $\kg \ast \DPsi \kf$
have joint compact support $Q_2$,  for any $\sPsili$.

For the next step we recall that the convolution between a measure $\mu$
with a test function $k$ can be determined pointwise by  $ \mu \ast k(x) = \mu(T_x k \checkm)$,
with $k \checkm(x) = k(-x)$. Thus
$$(\kg \ast\kf -\kg \ast \DPsi \kf)(x)=(\kf - \DPsi \kf) (T_x \kg \checkm)$$
in the pointwise sense.
In fact, it is valid uniformly over compact
sets, but outside of $Q_2$ all the functions are zero anyway, hence
we have uniform convergence and joint compact support.
Since the weight $w$ is bounded over $Q_2$ it is then
clear that one has
$$\lim_{|\Psi| \to 0}\normta{\kg \ast \kf - \kg\ast\DPsi \kf}\Liwsp = 0.$$
In other words, one can find $\delta > 0$ such  that  for $\sPsi \leq \delta$ ($\leq 1$) one has
\begin{equation} \label{kgkfest1}
 \normta{  \kg \ast \kf - \kg \ast \DPsi \kf} \Liwsp  < \veps/4.
\end{equation}
We also have to control the transition to the discretized form:
$$ \normta{\kg \ast \DPsi \kf - g \ast \DPsi f} \Liwsp \leq
 \normta{\kg \ast \DPsi \kf - \kg \ast \DPsi f} \Liwsp +
  \normta{\kg \ast \DPsi f - g \ast \DPsi f} \Liwsp,
$$
which thanks to (\ref{LiwastMiw}) can be continued by the estimate
$$
\leq \normta { \kg  \ast  \DPsi (\kf-f) } \Liwsp +
   \normta { (\kg -g) \ast \DPsi f} \Liwsp
$$
$$
\leq  \normta \kg \Liwsp \cdot \normta { \DPsi (\kf-f) } \Miwsp
+
 \normta {(\kg -g)} \Liwsp \cdot \normta{\DPsi f}\Miwsp,
$$
and finally, using the normalization assumption $\normta g \Liwsp = 1 = \normta f \Liwsp$, we have:
$$
\leq 2 \normta g \Liwsp \cdot C_1 \normta {\kf-f} \Liwsp
+
 \normta {\kg -g} \Liwsp \cdot C_1 \normta{f} \Liwsp \leq 3 C_1 \eta.
$$
By the choice of $\eta$  we get:
\begin{equation} \label{lastest4}
\normta{\kg \ast \DPsi \kf - g \ast \DPsi f} \Liwsp < \veps/4.
\end{equation}
Combining the estimates (\ref{firstest1}),(\ref{kgkfest1}) and (\ref{lastest4}) the claim, i.e.\ formula (\ref{convappr02}) in the theorem is verified.
\end{proof}

Next we will use the Cohen-Hewitt factorization theorem to show
that a similar result is true for the action on a Banach module.
Since we are dealing with a commutative situation and functions
and measures over $\Rdst$ we keep the order and write convolution
from the right.

\begin{theorem} \label{Liwmod1}
Any  minimal TMIB Banach space of tempered distributions $\BspN$  is an essential Banach module over some Beurling algebra $\LiwRdN$. Moreover, one
has for any  $g \in \Bsp$ and $k \in \LiwRd$:
\begin{equation}\label{convappr01}
  \| g \ast k - g \ast \DPsi k \|_\Bsp \to 0 \quad \mbox{for} \,\, \sPsitoz.
\end{equation}
\end{theorem}
\begin{proof}
Since  $\BspN$ is an essential Banach module over the Banach convolution
algebra $\LiwRdN$ we can apply the Cohen-Hewitt factorization Theorem,
(\citeX{hero70}, Chap.32), i.e. any $g \in \Bsp$ can be written as
$ g = g_1 \ast h$, with $h \in \LiwRd$.
Using  the associativity law for Banach modules we obtain therefrom:
\begin{equation}\label{Banmodest05}
  \Bnorm {g \ast k - g \ast \DPsi k }
  = \Bnorm {g_1 \ast h \ast k - g_1 \ast h \ast k}
  \leq \Bnorm {g_1} \normta{ h \ast k - h \ast \DPsi k} \Liwsp \to 0
\end{equation}
as $\sPsitoz$, according to Theorem $\nth$ \ref{discrconvBeurl1}.
\end{proof}


\section{The Main Result}

We are now ready to formulate our main result:
\begin{theorem}
Given a  minimal tempered standard space $\BspN$ on $\Rdst$, and any
 $g \in \ScRd$ with $\intRd g(x)dx = \hatg(0) \neq 0$,  the set
$$ S(g) := \{ T_x \Strho g  \suth  x \in \Rdst, \rho \in (0,1] \} $$
is total in $\BspN$, i.e.\ the finite linear combinations
are dense.
\end{theorem}
\begin{proof}
The claim requires to find, for any given  $ f \in \Bsp$ and
$\epso$, some finite linear combination $h$ of elements from $S(g)$
such that
\begin{equation} \label{epsest1}
\| f-h\|_\Bsp < \veps.
\end{equation}
We will verify something slightly stronger: Given $f, \epso$
there exists $\rho_0 < 1$ such that
for any (fixed) $\rho \in (0,\rho_0]$ one can find a finite
set $(x_i)_{i \in F}$ and coefficients $(c_i)_{i \in F}$
such that  $ h = \sum_{i \in F} c_i T_{x_i}g_{\rho}$
satisfies (\ref{epsest1}).

This approximation will be achieved in four steps:
 \begin{enumerate}
  \item By the  density of $\ScRd$ in $\BspN$ and the density
  of compactly supported functions in $\ScRd$ (in the Schwartz
  topology), we can find
  some $k \in \DRd \subset \Bsp \cap \LiwRd$ with
  \begin{equation}\label{BDRdapp}
   \|f-k\|_\Bsp < \epsfo.
  \end{equation}

    \item In the next step we apply Lemma \ref{essModChar1} for the
    specific approximate unit $(g_\rho)_{\rho \to 0}$, according
    to Lemma \ref{StrhoDirac1}. Hence there exists some $\rho_0$
    such that for any $\rho \in (0,\rho_0]$ one has
    \begin{equation} \label{convappr4}
  \| g_\rho \ast k - k \|_\Bsp     <    \epsfo.
  \end{equation}
  Let us fix one such parameter $\rho$ for the rest.

  \item The final step is the discretization of the
  convolution $g_\rho \ast k$, by replacing $k$ by some finite, discrete
  measure in $\MiwRd$, by applying Theorem \nnth \ref{Liwmod1}
  with $g = g_\rho$
  and $k \in \DRd$. By choosing $\delta_0>0$ properly we can guarantee
  that    $|\Psi| \leq \delta_0$ implies
   \begin{equation}\label{convDPsi}
    \| k \ast g_\rho -  (\DPsi k) \ast g_\rho\|_\Bsp <  \epsfo.
  \end{equation}

  Note that
   \begin{equation} \label{thirdest1}
   h = (\DPsi k) \ast g_\rho = \sum_{i \in I} c_i \delta_{x_i} \ast g_\rho
    = \sum_{i \in F} c_i  T_{x_i}g_\rho
   \end{equation}
  has the required form, because $ F = \{ i \in I \suth \supp(k) \cap \supp(\psi_i) \neq \emptyset\}$
  is a finite set, due to the compactness
  of $\supp(k)$. It depends only on  $\supp(k)$ and $\Psi$.
 \item Combining the estimates  (\ref{BDRdapp}), (\ref{convappr4}) and
(\ref{convDPsi}), we have for the given choice
$$
\normB {f - h}  \leq \normB {f -k} + \normB {k - g_\rho \ast k} +
\normB {k \ast g_\rho - h} \leq 3 \, \varepsilon/4,
$$
i.e.\  we have obtained  the desired estimate:
\begin{equation}\label{finest5}
  \| f - \sum_{i \in F} c_i  T_{x_i}g_\rho \|_\Bsp
  =  \| f - h\|_\Bsp <   \varepsilon,
\end{equation}
and the proof is complete.
\end{enumerate}
\end{proof}

\begin{remark}
There is some freedom for the  choice of the points $x_i$.
Their density depends on the translation behaviour of $g_\rho$ within $\BspN$. It is not  obvious to find the optimal choice, requiring minimal density of these points, combined with a good robustness of the approximation.
If $\rho$ is close to zero, then one expects that the finite family
has to be chosen very densely within $\supp(k)$. On the other hand,
working with relatively large $\rho$, which appears to be better in {\it this
respect}, the error $\| k - g_\rho \ast k\|_\Bsp$ will become larger.
\end{remark}

\section{Application to concrete cases}

This section contains essentially three parts. In the first part we collect
a few basic facts about weighted spaces. These will be used in the sequel
to convince the reader that the current setting includes all the cases
which are covered by the paper \citeX{ka19-1}, but in fact many more.
This will be explained in the second part of this section. A short
subsection is devoted to the case of Shubin classes $\QsRd$. Further
indication of the richness of examples is given in the final subsection.

\subsection{Weighted spaces, basic properties}

First let us summarize a few facts concerning function spaces,
in particular weighted $\Lpsp$-spaces over $\Rdst$, which are
the prototypical examples of MINTSTAs (resp. TMIBs).

Going back to the classical papers \citeX{ed59}, \citeX{ga69}, \citeX{fe79}, and \citeX{gr07}, let us recall a few general facts about translation invariant function spaces. Recall that two weights
$m_1$ and $m_2$  are called {\it equivalent} (we write $m_1 \approx m_2$)
if for some $C > 0$
      \begin{equation}\label{equivl}
 C \inv m_1(x) \leq m_2(x) \leq C m_1(x), \quad \forall x \in \Rdst.
      \end{equation}

\begin{lemma} Let $p \in [1,\infty)$ be given. \newline
\noindent
(1) A weighted $\Lpsp$-space $\NSPB \Lpmsp$ is translation invariant
    if and only $m$ is moderate;

\noindent
(2) For any $ k \in \CcG$ the function
   $ x \mapsto \|k\|_{m,p}  = \normta {k m}  {\Lpsp} $ is
   equivalent to the weight function $m$;

\noindent
(3) 
Any moderate weight function is equivalent to a continuous one;

\noindent  (4) 
Two spaces $\Lsp^{p_1}_{m_1}$ and  $\Lsp^{p_2}_{m_2}$
      are equal if and only if $p_1 = p_2$ and $m_1 \approx m_2$.
\end{lemma}

The following lemma is a consequence of the main results of \citeX{fe90},
choosing $p=2$ there.
\begin{lemma} \label{extFTLtwit}
The extended Fourier transform 
maps $\LtKats$ onto $\FLtKats$. In particular, one has Fourier invariant
spaces of the form $\LtKats$ if and only if $m_1  \approx  m_2$.
\end{lemma}

\subsection{Deducing Katsnelson's results}

The spaces considered by  Katsnelson in \citeX{ka19-1} are of the
form  $\Bsp = \LtKats$, with their natural norm.
We do not have to repeat the {\it technical conditions} made
in the paper \citeX{ka19-1}, but rather summarize the relevant
consequences of the setting described in that paper which
allow us to demonstrate that the setting chosen for the
current manuscript covers 
the cases described in Katsnelson's paper:
%
%
\begin{enumerate}
  \item $\BspN$ is continuously embedded into $\LtRN$;
  \item $\BspN$ is a Banach space, in fact even a Hilbert space;
  \item The spaces are invariant under translation and modulations.
\end{enumerate}

\noindent
The argument to be used next is taken from Lemma 2.2. 
of \cite{fegu90}.
\begin{lemma} \label{KatsinLt}
Given a space of the form $\LtKats$
with two continuous, moderate weights $m_1$ and $m_2$
%
one has a continuous  embedding into $\LtRdN$ if and only if
both  $m_1$ and $m_2$ are bounded away from zero, which in turn
is equivalent to the assumption that both
$$ \Ltmisp \hookrightarrow \LtRd \qandq \Ltmtsp \hookrightarrow \LtRd.$$
\end{lemma}

\noindent
In conclusion we have the following observation:

{\it The setting described in the paper \citeX{ka19-1} is exactly
equivalent to the assumptions made in Lemma \ref{KatsinLt}.
Obviously these spaces are then Banach spaces of tempered
distributions in the sense of our Definition \ref{mintempstanddef} and hence
all the results of the current paper or facts in \cite{dipivi15-1}
apply to that situation (e.g.\ Prop. 3.4). }

The interested reader is referred to \citeX{fe90} for details
in this direction.



\begin{proposition} \label{SobKats2}
For $m_1(x) = \japx^s = m_2(x) $, $s \in \Rst$ the
corresponding spaces $\LtKats$ are Fourier invariant, as the intersection
of a Sobolev space with the corresponding weighted $\Ltsp$-space.

They can also be identified with the so-called {\it Shubin classes}
$\QsRdN$, characterized as the Banach spaces of all tempered
distributions with a short-time Fourier transform in $\LtvsTFd$.
\end{proposition}

\begin{remark} \label{ShubHerm1}
For $d = 1$ the spaces $\Qssp(\Rst)$ coincides with
 a space of tempered distributions having Hermite coefficients
in a weighted $\ltsp$-space (with polynomial weight of the order $s/2$).
\end{remark}

In order to show that the Shubin classes $\QsRdN$ (for $s \geq 0$) are
covered even in the setting of Katsnelson's paper we have to shortly
recall the concept of {\it modulation spaces} (see \cite{fe03-1,fe06}).

A meanwhile widely used  variant of modulation spaces (skipping
many technical details) are the space $\nspb \MpqmRd$, which are those
tempered distributions which have an STFT (Short-Time Fourier Transform)
belonging to a (moderately) weighted mixed-norm space (with two
independent parameters $p,q$). This STFT of $\sigma \in \ScPRd$
can be defined for any (say real-valued) Schwartz window $g \in \ScRd$ by:
\begin{equation} \label{Vgsigma1}
 V_g(\sigma)(x,y)  =  \sigma (M_y T_x g), \quad (x,y) \in \TFd.
 \end{equation}

The choice $m(x,y) = \japarg{y}^s$  then gives the classical modulation
space $\nspb \MspqRd$. For $p=q$ and radial symmetric weights of the
form $m(x,y) = v_s(z) = (1+x^2 + y^2)^{s/2}$ one has the (Fourier
invariant) modulation spaces $\MpvsRd$ (see \citeX{gr01}). For more
information on modulation spaces see \citeX{beok20}, \citeX{fe03-1},
and \citeX{fe06}.

In order to verify that the Shubin classes (see also \citeX{lura11})
 are special cases of Katsnelson's paper we need the following simple observations:
\begin{lemma} \label{KatsShub1}

\noindent (i)
Given two modulation spaces $\Msp^{p,q}_{m_1}$  and $\Msp^{p,q}_{m_2}$,
     one has $\Msp^{p,q}_{m_1} \cap \Msp^{p,q}_{m_2} = \Msp^{p,q}_{m}$, with      $m = \max(m_1,m_2)$, and equivalence of the corresponding natural norms.

\noindent (ii)
For polynomial weights of the form $m_1(x,y) = \japx^s$
   and $m_2(x,y) = \langle y \rangle^s $ for some $s \in \Rst$,
   one has:  \newline  $\max(m_1,m_2)(x,y)  \sim v_s(z) = v_s(x,y)$.
\end{lemma}
\begin{proof}
For the second statement one just has to observe that for $z = (x,y)$ one
has $\max( \japarg{x}^s, \japarg{y}^s) \sim \japarg{z}^s$, using the fact
that $\max{(|x|,|y|)} \geq |z|$, since obviously
$$
\max(|x|,|y|) \leq  |z| \leq |x| + |y| \leq 2 \max(|x|,|y|), \quad z = (x,y).
$$
\end{proof}

\begin{remark}
As a final remark let us observe that $\QsRd \subset \LtRd$ if and only
if $s \geq 0$. Thus the results of \citeX{ka19-1} apply only for these
spaces, while our general results allow the full range $s \in \Rst$, including their dual spaces, since $\QsRd^* = \Qsp_{-s}(\Rdst)$.
\end{remark}

\section{Summary}

In conclusion this paper provides extensions of
 the main result of \cite{ka19-1} in the following directions:
\begin{enumerate}
  \item The Gauss function can be replaced by
    any Schwartz function with non-zero integral;

  \item The results are valid for $\Rdst$, for any $d \geq 1$;

  \item We abolish the assumption that   $\BspN$
    is a Hilbert space, as well as the rather restrictive property that it should be contained in $\LtRdN$;

  \item Our main result applies to an abundance of function spaces for which
    such completeness statements can be shown; we just list particular examples;

  \item As a benefit we establish a connection to the so-called
    Shubin classes and show that the completeness statement is
    also valid for their dual spaces.
 \end{enumerate}

In order to avoid the use of ultra-distributions and a generalized
Fourier transform in the sense of such ultra-distributions we decided
to work with the well-known setting of {\it tempered distributions} in the
sense of Schwartz, i.e.\ with $\ScPRd$, the dual of the Schwartz
space $\ScRd$ of rapidly decreasing functions. Extension to the
setting of ultra-distributions are no problem but will be discussed
elsewhere.  Also the setting of \citeX{fe77-2} (Theorem 2.2) should be helpful in this respect.

\section{Acknowledgement}

This paper was prepared during the visit of the
 second author to the NuHAG workgroup at the University of Vienna,
supported by an Ernst Mach Grant-Worldwide Fellowship (ICM-2019-13302)
 from the OeAD-GmbH, Austria. The second author is very grateful to
Professor Hans G. Feichtinger for his guidance, for hosting and
 arranging excellent research facilities at the University of Vienna.
The second author is also grateful to the NBHM-DAE (0204/19/2019R\&D-II/10472) and Indian  Institute of Science,
Bangalore for allowing the academic leave.

\bibliographystyle{abbrv}  

\end{document}